\documentclass[a4paper,latexcad,twoside,11pt]{article}
\usepackage[mathscr]{euscript}
\usepackage{xcolor}
\usepackage{amsthm}
\usepackage{amsmath}
\usepackage{stmaryrd} 
\usepackage{amssymb}
\usepackage{makeidx}
\usepackage{indentfirst}
\usepackage{amsfonts}
\usepackage{bm}
\usepackage{cases}
\usepackage{graphicx} 
\usepackage{caption}   
\usepackage{subcaption}
\usepackage{hyperref}
\usepackage{enumitem}
\usepackage{float}
\usepackage{mathtools} 
\usepackage[a4paper]{geometry}
\title{The DP color function of joins and vertex-gluing of hypergraphs}
\author{ }

\date{}

\newtheorem{theorem}{Theorem}
\newtheorem{lemma}[theorem]{Lemma}
\newtheorem{corollary}[theorem]{Corollary}
\newtheorem{proposition}[theorem]{Proposition}
\newtheorem{definition}{Definition}

\newcommand{\brk}[1]{\llbracket #1 \rrbracket}

\begin{document}
	
	\author{Ruiyi Cui$^1$, 
		Liangxia Wan$^1$\thanks{
			Emails: lxwan@bjtu.edu.cn. },  
		Fengming Dong$^2$\thanks{
			Email: fengming.dong@nie.edu.sg}
		\\
		\small $^1$School of Mathematics and Statistics,
		Beijing Jiaotong University, Beijing $100044$, China
		\\
		\small $^2$National Institute of Education, Nanyang Technological University, Singapore
	}
	
	\maketitle
	
	\begin{abstract}
		The DP color function $P_{DP}(\mathcal{H}, k)$ of a hypergraph $\mathcal{H}$ is defined as the minimum possible value of $P_{DP}(\mathcal{H}, \mathcal{F})$ taken over all $k$-fold covers $\mathcal{F}$ of $\mathcal{H}$. It serves as an extension of the classical chromatic polynomial. In this paper, we investigate the DP color function for two constructions: the join of a hypergraph with a complete graph and vertex-gluings of certain hypergraphs. Both constructions extend some results of Becker et al.
	\end{abstract}
	\vskip 2mm
	\noindent{\small{\bf Mathematics Subject Classiﬁcation} 05C15, 05C30, 05C65}
	\vskip 2mm
	\noindent{\bf keywords} Chromatic polynomial$\cdot$ Hypergraph coloring$\cdot$ DP-coloring$\cdot$ DP color function

	\section{Introduction}
	In this paper, each hypergraph is nonempty, finite and simple unless otherwise noted. 
	Let $\mathbb{N}$ and $\brk{k} $ denote the set of all positive integers and the $k$-element set $\{1,2,..,k\}$, respectively. 
	A hypergraph $\mathcal{H} = (X, E)$ consists of a finite non-empty set $X$ and a subset $E$ of $2^X$ (i.e., the power set of $X$) with $|e| \geq 2$ for each $e \in E$. 
	A hypergraph $\mathcal{H}$ is \textit{$r$-uniform} if  $|e| = r$ for each edge $e \in E$. 
	A hypergraph is \textit{linear} if every pair of distinct edges intersects in at most one vertex. 
	A cycle $C$ of length $p$ in $\mathcal{H}$ is a subhypergraph consisting of $p$ distinct vertices $v_1, e_2, ..., v_p$ and $p$ distinct edges $e_1, e_2, \ldots, e_p$ such that $v_{i-1}, v_{i} \in e_i$ for each $i \in \brk{p}$ (indices are taken modulo $p$). 
	A \textit{hypertree} is a connected linear hypergraph without cycles. 
	A \textit{unicyclic hypergraph} is a connected hypergraph containing exactly one cycle. 
	A hypergraph $\mathcal{H}= (V,E)$ is \textit{complete} if $E = P(V) \setminus \{ \emptyset \}$ where $P(V)$ denotes the powerset of $V$. A subset $S\subseteq E(\mathcal{H})$ is also called a {\it partial hypergraph} of $\mathcal{H}$. 
	For a hypergraph $\mathcal{H}$, $\mathcal{H} \vee K_1$ is the hypergraph obtained from $\mathcal{H}$ by adding a new vertex $u$ and adding new edges $\{u,v\}$ for all vertices $v$ in $\mathcal{H}$.
	
	Helgason \cite{HE06} proposed the hypergraph chromatic polynomial for the first time in 1972. 
	For any positive integer $k$, a \textit{proper $k$-coloring} of the hypergraph $\mathcal{H}$ is a map $f\colon V(\mathcal{H})\to \brk{k}$ such that, for every edge $e\in E(\mathcal{H})$, there exist two distinct vertices $u,v\in e$ satisfying $f(u)\neq f(v)$. The number of proper $k$-colorings of $\mathcal{H}$ is a polynomial $P(\mathcal{H},k)$ in the variable $k$, of degree $|V(\mathcal{H})|$ in $k$, called \textit{chromatic polynomial} of $\mathcal{H}$.
	In the 1970s, the vertex coloring was independently generalized to the list coloring by Vizing \cite{VI76} on one side, and Erdős, Rubin, and Taylor \cite{ERT79} on the other. Later, in 2015, Dvořák and Postle \cite{DP18} generalized the list coloring to the DP-coloring (also called correspondence coloring) to prove that every planar graph without cycles of length $4$ to $8$ is $3$-choosable.  The DP-coloring has been extensively studied \cite{BHKMMSTW22,DY22,KM21,LLYY19, MP22,MU22}. Kaul and Mudrock \cite{KM21} introduced a DP color function of a graph. Becker et al consider the DP color function of joins and vertex-gluings of graphs and they obtained the following results.
	\begin{theorem}\cite{BHKMMSTW22}\label{result 1}
		Suppose \(P_{DP}(K_p \vee G, m) = P(K_p \vee G, m)\) whenever \(m \geq N\). Then, \(P_{DP}(K_{p+1} \vee G, s) = P(K_{p+1} \vee G, s)\) whenever \(s \geq N + 1\). Consequently, \(\tau_{DP}(K_{p+1} \vee G) \leq \tau_{DP}(K_p \vee G) + 1\).
	\end{theorem}
	
	\begin{theorem}\cite{BHKMMSTW22}\label{result 2}
		For any \(k, m \in \mathbb{N}\), \(P_{DP}(K_1 \vee C_{2k+1}, m) = P(K_1 \vee C_{2k+1}, m)\), and
		\[
		P_{DP}(K_1 \vee C_{2k+2}, m) = \begin{cases} 
			0 & \text{if } m \in [2] \\
			3 & \text{if } m = 3 \\
			P(K_1 \vee C_{2k+2}, m) & \text{if } m \ge 4.
		\end{cases}
		\]
	\end{theorem}
	\begin{theorem}\cite{BHKMMSTW22}\label{result 3}
		Suppose that \(G_1, \dots, G_n\) are vertex disjoint graphs for some \(n \geq 2\) and \(G \in \bigoplus_{i=1}^n (G_i, 1)\). Then
		\[
		P_{DP}(G, m) \leq \frac{\prod_{i=1}^n P_{DP}(G_i, m)}{m^{n-1}}.
		\]
	\end{theorem}

	 Bernshteyn and Kostochka \cite{BK19} introduced the DP-colorings for hypergraphs. We introduced the DP color functions of hypergraphs \cite{CWD26}. 	
	Let $\varphi \colon X \rightharpoonup Y$ be a \textit{partial map} $\varphi$ from a subset of $X$ to $Y$. $[X \rightharpoonup Y]$ denotes the set of all partial maps. For $k \in \mathbb{N}^+$, consider a family $\mathcal{F} \subseteq [X \rightharpoonup \brk{k} ]$ of partial maps. The domain $\operatorname{dom}(\mathcal{F})$ of $\mathcal{F}$ is defined by
	$$
	\operatorname{dom}(\mathcal{F}) \coloneqq \bigl\{ \operatorname{dom}(\varphi) : \varphi \in \mathcal{F} \bigr\}.
	$$
	
	Then each $\varphi \in \mathcal{F}$ is a subset of the Cartesian product $X \times \brk{k} $. Thus, we can view $\mathcal{F}$ as a hypergraph on $X \times \brk{k} $ and view $\text{dom}\operatorname{dom}(\mathcal{F})$ as a hypergraph on $X$.
	A function $f\colon X \to \brk{k}$ avoids a partial map $\varphi\colon X \rightharpoonup \brk{k}$ if $\varphi \nsubseteq f$. Given a family $\mathcal{F} \subseteq [X \rightharpoonup \brk{k}]$ of partial maps, a function $f\colon X \to \brk{k}$ that avoids every $\varphi \in \mathcal{F}$ is called a \textit{$(k,\mathcal{F})$-coloring} (or simply an \textit{$\mathcal{F}$-coloring}).
	
	A \textit{$k$-fold cover} of $\mathcal{H}$ is a family of partial maps $\mathcal{F} \subseteq [X \rightharpoonup \brk{k}]$ such that $\operatorname{dom}(\mathcal{F}) \subseteq V(\mathcal{H})$ and if $\operatorname{dom}(\varphi) = \operatorname{dom}(\psi)$ for distinct $\varphi, \psi \in \mathcal{F}$, then $\varphi \cap \psi = \emptyset$.
	Let $\mathcal{F}_e = \{ \varphi \in \mathcal{F} : \text{dom}(\varphi) = e \}$ and let $\mathcal{F} \setminus \mathcal{F}_e = \{ \varphi \in \mathcal{F} : \text{dom}(\varphi) \neq e \}$. If $\mathcal{F}$ is a $k$-fold cover of $\mathcal{H}$, then, $\mathcal{F} \setminus \mathcal{F}_e$ is a $k$-fold cover of $\mathcal{H} - e$. If $\mathcal{F}$ is a $k$-fold cover of a hypergraph $\mathcal{H}$, then, for each edge $e$ in $\mathcal{H}$,
	$$
	|\mathcal{F}_e| = \bigl| \{ \varphi \in \mathcal{F} : \operatorname{dom}(\varphi) = e \} \bigr| \leq k.
	$$
	
	Throughout this paper, all $\mathcal{F}$ are perfect, i.e., $|\mathcal{F}_e|=k$ for each $e\in E(\mathcal{H})$, unless otherwise stated. A hypergraph $\mathcal{H}$ on a set $X$ is \textit{$k$-DP-colorable} if each $k$-fold cover $\mathcal{F}$ of $\mathcal{H}$ admits a $(k, \mathcal{F})$-coloring $f: X \to \brk{k}$.
	The \textit{DP color function threshold} of $\mathcal{H}$, denoted by $\tau_{DP}(\mathcal{H})$, is the smallest $N \geq \chi(\mathcal{H})$ such that $P_{DP}(\mathcal{H},k)=P(\mathcal{H},k)$ whenever $k \geq N$. 
	Let $\tau_{DP}(\mathcal{H}) = \infty$ if $P_{DP}(\mathcal{H},k)>P(\mathcal{H},k)$ for infinitely many $k$.
	Let $\mathcal{H}$ be a hypergraph with $E(\mathcal{H})=\{e_j: j\in \brk{m}\}$. A \textit{natural $k$-cover} of $\mathcal{H}$ is denoted by
	$$
	\iota_{\mathcal H}(k)
	=\left \{\varphi^{(j)}_i:  i\in \brk{k},  j\in \brk{m}\right \},
	$$ 
	where $\varphi^{(j)}_i$ is the map: $e_j\mapsto \{i\}$. 
	If $\mathcal I$ is the
	set of $\iota_{\mathcal H}(k)$-colorings of $\mathcal H$ and $\mathcal J$ is the set of proper $k$-colorings of $\mathcal H$, then the function
	$f: \mathcal J\rightarrow \mathcal I$ given by
	$$ f(c) = \{(v,c(v)): v\in V
	(\mathcal H)\}$$
	is a bijection between the $\iota_{\mathcal H}(k)$-colorings  of $\mathcal H$ and its proper $k$-colorings.
	
	Suppose that $\mathcal F$ is a $k$-fold cover of a hypergraph $\mathcal H$ for $k\in \mathbb{N}$. Let $P_{DP}(\mathcal H,\mathcal F)$ be the number of $\mathcal F$-colorings of $\mathcal H$. The \textit{DP color function} denoted by $P_{DP}(\mathcal H,k)$ is the minimum value of $P_{DP}(\mathcal H,\mathcal F)$ for all $k$-fold covers $\mathcal{F}$ of $\mathcal{H}$. It is clear that for any hypergraph $\mathcal H$ and $k\in \mathbb{N}$, $P_{DP}(\mathcal H,k)\le P(\mathcal H,k)$.
	
	The remainder of this paper is organized as follows. In Section~\ref{sec:prelim}, we recall some known results that will be needed in our proofs. Section~\ref{sec:Join} is devoted to establishing properties of the DP color function for the join of a hypergraph with a complete graph. In particular, we extend Theorems \ref{result 1} and \ref{result 2} to the hypergraph case. In Section~\ref{sec:Glu}, we introduce the definition of vertex-gluings of certain hypergraphs and present related theorems. Specifically, we extend Theorem \ref{result 3} to the hypergraph case.
	
	\section{Preliminaries}\label{sec:prelim}
	In this section, we review several results to be utilized in the subsequent proof.

	\begin{lemma}\label{lem:2-1}
		If $\mathcal T_{m}^{r}$ is a $r$-uniform hypertree with $m$ edges, where $r\geq 2$ and $m\geq 0$, then $$P(\mathcal T_{m}^{r},k)=k(k^{r-1}-1)^{m}.$$
	\end{lemma}
	
	Borowiecki and Łazuka provided the chromatic polynomials of $r$-uniform unicyclic hypergraphs with $m$ edges as follows \cite{BŁ07}.
	\begin{lemma}\label{lem:2-3}
		Let $\mathcal H$ be a linear hypergraph. $\mathcal H$ is a $r$-uniform unicyclic hypergraph with $m + p$ edges and a cycle of length $p$ if and only if
		$$P(\mathcal H,k)=(k^{r-1}-1)^{m+p}+(-1)^{p}(k-1)(k^{r-1}-1)^{m}$$
		where $r\geq 3,m\geq 0$ and $p\geq 3$.
	\end{lemma}
	
	\begin{lemma}\cite{ZD17}\label{K_1+H}
		For any hypergraph $\mathcal{H}$, $P(K_1 \vee \mathcal{H},k) = kP(\mathcal{H},k-1)$.
	\end{lemma}

	\begin{lemma}\cite{CWD26}\label{P_DP(C)}
		Suppose $\mathcal H$ is a unicyclic linear $r$-uniform hypergraph with $m+p$ edges, and its unique cycle is of 
		length $p$ for $p\ge 3$, $r\ge 3$ and $m\ge 0$. 
		Then, 
		\begin{enumerate}
			\item[(1)]  If $p$ is odd, then 
			$$
			P_{DP}(\mathcal H,k)=P(\mathcal H,k).
			$$
			\item[(2)] 
			If $p$ is even, then 
			$$P_{DP}(\mathcal H,k)=(k^{r-1}-1)^{m+p}+(-1)^{p+1}(k^{r-1}-1)^{m}.$$
		\end{enumerate} 
	\end{lemma}
	
	\begin{lemma}\cite{CWD26}\label{lem_nature}
		Suppose that $\mathcal H$ is a hypergraph with $m$	edges and has no cycles for $m\ge 0$. 
		Suppose that $\mathcal F=\{\varphi_i^{(j)}: i\in \brk{k}, j\in \brk{m}\}$ is a $k$-fold cover
		of $\mathcal H$ for $k\ge 2$. 
		Then there is a bijection between $\mathcal F$ and $\iota_{\mathcal {H}}$.
	\end{lemma}
	
	\begin{lemma}\cite{CWD26}\label{lem1}
		Suppose that $\mathcal H$ is a hypergraph and $e = \{v_{i}: i \in \brk{r}\} \in E(\mathcal{H})$, where $r \ge 2$ and $c(\mathcal{H}-e)=r-1$.
		Let $v_1$ and $v_2$ be the two vertices in $e$ which are in the same component of $\mathcal{H}-e$. 
		Let $S_{k}^{(i_1,i_2,\ldots,i_r)}$ be the set of proper $k$-colorings of $\mathcal H-e$ that color $v_{l}$ by $i_l\in \brk{k}$ for all $l\in \brk{r}$, where $k\ge 2$. Then for any given $i_l\in\brk{k}$
		$$
		t_{1}=\frac{1}{k}(P(\mathcal H-e,k)-P(\mathcal H,k))
		$$
		and
		$$t_{2}=\frac{1}{k^{r-1}(k-1)}(k^{r-2}P(\mathcal H,k)+(1-k^{r-2})P(\mathcal H-e,k)),
		$$ 
		where  $t_1:=|S_{k}^{(i_1,i_2,\ldots,i_r)}|$ 
		when $i_1=i_2$, and $t_2:=|S_{k}^{(i_1,i_2,\ldots,i_r)}|$,
		when $i_1\ne i_2$.
	\end{lemma}
	
	\begin{lemma}\cite{CWD261} \label{prop1.1}
		Let $\mathcal H$ be a hypergraph. Suppose that $e = \{v_i : i \in \brk{n_e}\} \in E(\mathcal{H})$, where $n_{e} \ge 2$ and $c(\mathcal{H}-e)=n_e-1$.
		If $P(\mathcal{H}-e, k) < \frac{k^{n_e-1}}{k^{n_e-1}-1} P(\mathcal{H},k)$, then  
		$$
		P_{DP}(\mathcal H,k)<P(\mathcal H,k).
		$$
		
	\end{lemma}
	
	\section{Joins of graphs}\label{sec:Join}

	Throughout this section, let $\mathcal{H}$ be a hypergraph with vertex set $V(\mathcal H)=\{v_i:i\in\brk n\}$ and edge set $E(\mathcal H)=\{e_{p}:p\in\brk m\}$. 
	Let $\mathcal M=K_1\vee\mathcal H$, and let $\omega$ denote the vertex corresponding to the copy of $K_1$. 
	For each $i\in\brk n$, denote $e_{m+i}=\{\omega,v_i\}$. Let
	$$
	\mathcal F = \left\{ \varphi_j^{(i)}(e_{m+i}): i\in\brk n,\;j\in\brk k,\;\varphi_j^{(i)}(\omega)=j\right\} \cup \left\{ \psi_j^{(p)}(e_{p}): p\in\brk m,\;j\in\brk k \right\}
	$$
	be a $k$-fold cover of $M$.
	
	For each $j\in\brk k$, define $\varPhi_j(\mathcal M,\mathcal F)=\left\{\varphi_j^{(i)}(e_{m+i}):i\in\brk n\right\}.$
	We say that $\varPhi_j(\mathcal M,\mathcal F)$ is a \emph{level map} if for every $p\in\brk m$, there exists $q_p\in\brk k$ such that
	$$
	\psi_{q_p}^{(p)}(v_i)=\varphi_j^{(i)}(v_i)
	$$
	for every $v_i\in e_{p}$.
	
	For each $j\in\brk k$, let
	$$
	\varPsi_j(\mathcal{H},\mathcal{F})=\left\{\psi_q^{(p)}(e_{n+p}):\forall p\in\brk{m},\;\exists q\in\brk{k}\text{ with }\psi_{q_p}^{(p)}(v_i)=\varphi_j^{(i)}(v_i)\right\}
	$$
	and define
	$$
	\mathcal F_j=\mathcal F-\varPsi_j(\mathcal H,\mathcal F)-\left\{\varphi_q^{(i)}(e_i):i\in\brk n,\;q\in\brk k\right\}.
	$$
	Then $\mathcal F_j$ is the cover of $\mathcal H$ induced by fixing the color $j$ at $\omega$.

	\begin{lemma}\label{Join-lemma1}
		If $P_{DP}(\mathcal{H},k) = P(\mathcal{H},k)$ for some hypergraphs $\mathcal{H}$ and $k \in \mathbb{N}$, then 
		$$P_{DP}(K_1 \vee \mathcal{H},k+1) = P(K_1 \vee \mathcal{H},k+1).$$
	\end{lemma}
		\begin{proof}
		Suppose that $\mathcal{F}$ is a $(k+1)$-fold cover of $\mathcal{H}$. It is easily to know that for each $\mathcal{F}$-coloring of $\mathcal{H}$
		\begin{align}
			P_{DP}(K_1 \vee \mathcal{H},\mathcal{F})=\sum_{j = 1}^{k}P_{DP}(\mathcal{H},\mathcal{F}_j)\geq k P_{DP}(\mathcal{H},k)
	\end{align}
		Lemma \ref{K_1+H} implies that 
		\begin{align}
			kP_{DP}(\mathcal{H},k) = kP(\mathcal{H},k) = P(K_1 \vee \mathcal{H},k+1).
		\end{align}
		Together with $(1-2)$, we have that
		\begin{align}
			P_{DP}(K_1 \vee \mathcal{H},\mathcal{F})\geq P(K_1 \vee \mathcal{H},k+1).
		\end{align}
		
	It is obvious that
		\begin{align}
			P_{DP}(K_1 \vee\mathcal{H},k+1) \le P(K_1 \vee \mathcal{H},k+1)					\end{align}
		Since $\mathcal{F}$ is arbitrary, the result  is deduced by applying with $(3-4)$.
	\end{proof}

The following result is a generalization of Theorem \ref{result 1} obtained by Becker et al.	
	\begin{theorem}\label{th1}
		If $P_{DP}(K_p \vee \mathcal{H},k) = P(K_p \vee \mathcal{H},k)$ whenever $k \geq N$, then for any $l \geq N+1$ 
		$$P_{DP}(K_{p+1} \vee \mathcal{H}, l) = P(K_{p+1} \vee \mathcal{H},l) \mbox{ and }\tau(K_{p+1} \vee \mathcal{H}) \le \tau(K_{p}\vee\mathcal{H}) + 1.$$
	\end{theorem}
	\begin{proof}
		Assume that $P_{DP}(K_p \vee \mathcal{H},k) = P(K_p \vee \mathcal{H},k)$ whenever $k \geq \mathbb{N}$. Lemma \ref{Join-lemma1} implies that 
		$$
		P_{DP}(K_1\vee(K_p \vee \mathcal{H}), k+1) = P(K_1\vee(K_p \vee \mathcal{H}), k+1).
		$$
		And, it is clear that 
		$$K_1\vee(K_p \vee \mathcal{H})  =K_{p+1} \vee \mathcal{H}.$$
		 Thus, the result immediately follows.
	\end{proof}
	
	\begin{lemma}\label{lem_odd}
		Suppose $\mathcal H$ is a unicyclic linear $r$-uniform hypergraph with $m+\lambda$ edges, and its unique cycle is of length $\lambda$ for $\lambda\ge 3$, $r\ge 3$ and $m\ge 0$. Then for any odd $\lambda\ge 3$ 
		$$P_{DP}(K_p\vee\mathcal{H},k) = P(K_p\vee\mathcal{H},k)\mbox{ and }\tau_{DP}(K_p\vee\mathcal{H}) = \chi(K_p\vee\mathcal{H})=p+2.$$
	\end{lemma}
	
	\begin{proof}
		From Lemma \ref{P_DP(C)}, when $\lambda$ is odd,
		$$
		P_{DP}(\mathcal H,k)=P(\mathcal H,k).
		$$
		
		Now, we will prove the desired by induction on $p$. When $p =1$, Lemma \ref{Join-lemma1} implies that for any $k \geq 2$ $$P_{DP}(K_1\vee\mathcal{H},k) = P(K_1\vee\mathcal{H},k).$$ The basic step is completed since $0 \le P_{DP}(K_1 \vee \mathcal{H},1) \le P(K_1 \vee \mathcal{H},1)=0 $.
		
		Now, assume that $p \geq 2$ and the desired result holds for all positive integers less than $p$. This means that for any $k \in \mathbb{N}$
		$$
		P_{DP}(K_{p-1}\vee\mathcal{H},k) = P(K_{p-1}\vee\mathcal{H},k).
		$$

		 Lemma \ref{Join-lemma1} implies that for any $k\ge 2$
		$$  
		P_{DP}(K_{p} \vee \mathcal{H}, k) = P_{DP}\bigl(K_1 \vee (K_{p-1} \vee\mathcal{H}), k-1\bigr) = P\bigl(K_1\vee (K_{p-1} \vee\mathcal{H}), k-1\bigr) = P(K_p \vee \mathcal{H}, k).
		$$
	 The desired result holds since
		$$
		0 \leq P_{DP}(K_p \vee \mathcal{H}, 1) \leq P(K_p \vee \mathcal{H}, 1) = 0.
		$$
		This completes the induction step.
		
	\end{proof}

		\begin{lemma}\label{lem_even_local}
		Suppose that $\mathcal H$ is a unicyclic linear $r$-uniform hypergraph with $m+\lambda$ edges, and its unique cycle $C$ has length $\lambda$, where $m\ge 0$, $r\ge 3$, and $\lambda\ge 4$ is even. Let $\mathcal M=K_1\vee\mathcal H$, and let $\mathcal F$ be an arbitrary $k$-fold cover of $\mathcal M$, where $k\ge 3$. Let $s$ be the number of level maps among $\{\varPhi_j(\mathcal M,\mathcal F):j\in\brk k\}$ where $\varPhi_j(\mathcal M,\mathcal F)$ is given above in this Section for each $j\in\brk k$.		
		For any $e\in E(C)$, let
		$$
		t_0=P(\mathcal H-e,k-1),
		$$
		$$
		t_1=\frac{1}{k-1}\left(P(\mathcal H-e,k-1)-P(\mathcal H,k-1)\right),
		$$
		and
		$$
		t_2=\frac{(k-1)^{r-2}P(\mathcal H,k-1)+(1-(k-1)^{r-2})P(\mathcal H-e,k-1)}{(k-1)^{r-1}(k-2)}.
		$$
		For each $j\in\brk k$, let $N_j$ denote the number of $\mathcal F$-colorings of $M$ containing $(\omega,j)$. Then if $\varPhi_j(\mathcal M,\mathcal F)$ is a level map, 
		$$
		N_j\ge t_0-(s-1)t_1-(k-s)t_2,
		$$
		
	\noindent	otherwise
		$$
		N_j\ge t_0-st_1-(k-2-s)t_2.
		$$
	\end{lemma}
	\begin{proof}
		Fix $j\in\brk k$, and let $\mathcal F_j$ given above in this section. Let $N_j$ denote the number of $\mathcal F$-colorings of $\mathcal M$ containing $(\omega,j)$. By the definition of $\mathcal F_j$, $N_j$ is precisely the number of $\mathcal F_j$-colorings of $\mathcal H$.
		
		Since $\mathcal H$ is linear, any edge outside $C$ on which $\varPhi_j(\mathcal M,\mathcal F)$ is non-level can be made level by a suitable relabeling of the colors in $\mathcal F$ at vertices outside $C$. 
		We choose an edge $e_1\in E(C)$ as follows. 
		If $\varPhi_j(\mathcal M,\mathcal F)$ is not a level map, choose $e_1$ so that $\varPhi_j(\mathcal M,\mathcal F)$ is non-level on $e_1$; if $\varPhi_j(\mathcal M,\mathcal F)$ is a level map, choose $e_1$ arbitrarily from $E(C)$. 
		Let $l$ be the number of maps on $e_1$ remaining in $\mathcal F_j$. Then
		$$
		l=k-1
		$$
		when $\varPhi_j(\mathcal M,\mathcal F)$ is a level map, whereas
		$$
		l\le k-2
		$$
		when $\varPhi_j(\mathcal M,\mathcal F)$ is not a level map.
		
		For each edge of $\mathcal H-e_1$ on which fewer than $k-1$ maps remain in $\mathcal F_j$, add partial maps so as to obtain a $(k-1)$-fold cover of $\mathcal H-e_1$, denoted by $\mathcal F_j^T$. 
		Since $e_1$ lies on the unique cycle of $\mathcal H$, the hypergraph $\mathcal H-e_1$ is acyclic. 
		
		Hence, by Lemma \ref{lem_nature}, there exists a bijection between $\mathcal F_j^T$ and the $(k-1)$-fold cover $\iota_{\mathcal H-e_1}$.
		Thus, we can relabel the colors in $\mathcal F_j^T$ according to this bijection so that $\mathcal F_j^T$ becomes a natural $(k-1)$-fold cover of $\mathcal H-e_1$, denoted by $\mathcal F_j^{T,*}$.
		
		We now construct $\mathcal F_j^{T,*}$. 
		Consider a component of $\mathcal H-e_1$. 
		Choose an edge $e_2$ in this component and a vertex $w\in e_2$. 
		We relabel $\mathcal F^{T}_j$ of this component as follows. 
		For each $\beta\in\brk{k}\setminus\{j\}$, there exists a unique partial map $\psi_{\beta}^{(2)}\in\mathcal F_j^T$ with domain $e_2$ such that
		$$
		\psi_{\beta}^{(2)}(w)=\beta.
		$$
		Define the relabeling map
		$$
		\rho_{\beta}^{(2)}:\psi_{\beta}^{(2)}(e_2)\longrightarrow\brk{k}\setminus\{j\}
		$$
		by
		$$
		\rho_{\beta}^{(2)}\bigl(\psi_{\beta}^{(2)}(v)\bigr)=\beta
		$$
		for every $v\in e_2$.
		Suppose that $v\in e_2\cap e$ for some edge $e\ne e_2$. For each $\beta\in\brk{k}\setminus\{j\}$, let $\psi_{\beta}^{(e)}$ be the unique partial map on $e$ satisfying
		$$
		\psi_{\beta}^{(e)}(v)=\psi_{\beta}^{(2)}(v).
		$$
		We then set
		$$
		\rho_{\beta}^{(e)}\bigl(\psi_{\beta}^{(e)}(v)\bigr)=\beta.
		$$
		Since the component is connected, linear, and acyclic, there is an unrelabelled edge that meets the relabelled part in exactly one vertex. 
		We apply the same procedure to this edge. 
		Repeating this process, after finitely many steps all colors in the component are relabelled. 
		Applying the same relabeling procedure independently to each component of $\mathcal H-e_1$, we obtain the natural $(k-1)$-fold cover of $\mathcal H-e_1$, denoted by $\mathcal F_j^{T,*}$. In particular, it has exactly
		$$
		t_0=P(\mathcal H-e_1,k-1)
		$$
		$\mathcal F_j^{T,*}$-colorings.
		
		The above procedure determines a relabeling map $\rho_v$ for the colors at each vertex $v\in V(\mathcal H-e_1)$. Since
		$$
		V(\mathcal H-e_1)=V(\mathcal H),
		$$
		the same relabeling is defined at every vertex of $e_1$. For each partial map $\psi$ remaining on $e_1$, define the corresponding relabeled partial map $\psi^*$ by
		$$
		\psi^*(v)=\rho_v\bigl(\psi(v)\bigr)
		$$
		for every $v\in e_1$. 
		Let $x$ and $y$ be the two vertices of $e_1$ that lie in the same component of $\mathcal H-e_1$. Let $a$ be the number of the $l$ relabeled partial maps $\psi^*$ on $e_1$ satisfying
		$$
		\psi^*(x)=\psi^*(y).
		$$
		If $\varPhi_j(\mathcal M,\mathcal F)$ is a level map, then each of the other $s-1$ level maps gives a distinct surviving partial map on $e_1$. After the relabeling, these maps assign the same color to $x$ and $y$. Hence
		$$
		a\ge s-1.
		$$
		If $\varPhi_j(\mathcal M,\mathcal F)$ is not a level map, then all the $s$ level maps give distinct surviving partial maps of this type. Hence
		$$
		a\ge s.
		$$
	
		By Lemma \ref{lem1}, there are exactly $t_1$ $\mathcal F_j^{T,*}$-colorings of $\mathcal H-e_1$ that agree with the fixed partial map on $e_1$ when $x$ and $y$ receive the same color. If $x$ and $y$ receive different colors, the number is exactly $t_2$.
		As $\lambda$ is even, by Lemmas \ref{lem:2-1} and \ref{lem:2-3} we have that
		$$
		t_1=\frac{\left((k-1)^{r-1}-1\right)^{m+\lambda-1}-(k-2)\left((k-1)^{r-1}-1\right)^m}{k-1}
		$$
		and
		$$
		t_2=\frac{\left((k-1)^{r-1}-1\right)^{m+\lambda-1}+\left((k-1)^{r-1}-1\right)^m}{k-1}.
		$$
	
		Suppose first that $\varPhi_j(\mathcal M,\mathcal F)$ is a level map. 
		In this case every edge of $\mathcal H$ has exactly $k-1$ maps remaining in $\mathcal F_j$, so no completion is required, and $l=k-1$. 
		Among the $k-1$ maps on $e_1$, exactly $a$ have the same color on $x$ and $y$, while the remaining $k-1-a$ have different colors on $x$ and $y$. Hence
		$$
		\begin{aligned}
			N_j
			&=t_0-at_1-(k-1-a)t_2\\
			&=t_0+a(t_2-t_1)-(k-1)t_2\\
			&\ge t_0+(s-1)(t_2-t_1)-(k-1)t_2\\
			&=t_0-(s-1)t_1-(k-s)t_2.
		\end{aligned}
		$$
		
		Now suppose that $\varPhi_j(\mathcal M,\mathcal F)$ is not a level map. 
		The cover $\mathcal F_j^T$ is obtained from the restriction of $\mathcal F_j$ to $\mathcal H-e_1$ by adding maps. 
		Thus every coloring admitted after the completion is also admitted by $\mathcal F_j$. 
		Among the $l$ surviving maps on $e_1$, exactly $a$ have the same color on $x$ and $y$. It follows that
		$$
		\begin{aligned}
			N_j
			&\ge t_0-at_1-(l-a)t_2\\
			&=t_0+a(t_2-t_1)-lt_2\\
			&\ge t_0+s(t_2-t_1)-(k-2)t_2\\
			&=t_0-st_1-(k-2-s)t_2,
		\end{aligned}
		$$
		where the second inequality follows from $a\ge s$, $l\le k-2$, and $t_2\ge t_1$. This completes the proof.
	\end{proof}
	
	\begin{corollary}\label{cor_even}
	Suppose that $\mathcal H$ is a unicyclic linear $r$-uniform hypergraph with $m+\lambda$ edges, and its unique cycle $C$ has length $\lambda$, where $m\ge 0$, $r\ge 3$, and $\lambda\ge 4$ is even. Let $\mathcal M=K_1\vee\mathcal H$, and let $\mathcal F$ be an arbitrary $k$-fold cover of $\mathcal M$, where $k\ge 3$. Let $s$ be the number of level maps among $\{\varPhi_j(\mathcal M,\mathcal F):j\in\brk k\}$ where $\varPhi_j(\mathcal M,\mathcal F)$ is given in this Section for $j\in\brk k$.
	\begin{align*}
		&P_{DP}(\mathcal M,k)\geq  \\
		&\min_{s\in\{0\}\cup(\brk{k}\setminus\{k-1\})}
		\left\{s\bigl(t_0-(s-1)t_1-(k-s)t_2\bigr)+(k-s)\bigl(t_0-st_1-(k-2-s)t_2\bigr)\right\}.
	\end{align*}
\end{corollary}
\begin{proof}
	Let $\mathcal F$ be an arbitrary $k$-fold cover of $\mathcal M$, and let $s$ be the number of level maps among $\{\varPhi_j(\mathcal M,\mathcal F):j\in\brk k\}$. For each $j\in\brk k$, let $N_j$ be the number of $\mathcal F$-colorings of $\mathcal M$ containing $(\omega,j)$. Since every $\mathcal F$-coloring of $\mathcal M$ contains exactly one color at $\omega$,
	$$
	P_{DP}(\mathcal M,\mathcal F)=\sum_{j\in\brk k}N_j.
	$$
	By Lemma \ref{lem_even_local}, there are $s$ indices $j$ for which
	$$
	N_j\geq t_0-(s-1)t_1-(k-s)t_2,
	$$
	and $k-s$ indices $j$ for which
	$$
	N_j\geq t_0-st_1-(k-2-s)t_2.
	$$
	Hence
	$$
	P_{DP}(\mathcal M,\mathcal F)\geq s\bigl(t_0-(s-1)t_1-(k-s)t_2\bigr)+(k-s)\bigl(t_0-st_1-(k-2-s)t_2\bigr).
	$$
	
	Moreover, $s\neq k-1$. Indeed, if $k-1$ of the maps $\varPhi_j(\mathcal M,\mathcal F)$ are level maps, then on each edge of $\mathcal H$ they correspond to $k-1$ distinct maps of the cover, and the remaining colors necessarily form the remaining map. Thus the last map is also a level map. Therefore,
	$$
	s\in\{0\}\cup(\brk{k}\setminus\{k-1\}).
	$$
	Since $\mathcal F$ was arbitrary, the proof is completed.
\end{proof}

The following theorem is a generalization of Theorem \ref{result 2} proved by Becker et al.
\begin{theorem}
	Suppose $\mathcal H$ is a unicyclic linear $r$-uniform hypergraph with $m+\lambda$ edges, and its unique cycle $C$ is of length $\lambda$ for $m\ge 0$, $r\ge 3$, $\lambda\ge 4$. Then for any $k \in \mathbb{N}$, 
	$$ P_{DP}(K_1\vee \mathcal{H}, k) = P(K_1\vee \mathcal{H}, k).$$
\end{theorem}
\begin{proof}
	When $k \geq 3$, from corollary \ref{cor_even}, $P_{DP}(\mathcal M,k)$ is at least
	$$
	\min_{s \in \{0\} \cup \{ \brk{k}-\{k-1\}\}}s(t_0- (s- 1)t_1- (k- s)t_2)+(k-s)(t_0- st_1- (k- 2-s)t_2)
	$$
	when $p$ is even. We now compete:
	\begin{align*}
		&s(t_0- (s- 1)t_1- (k- s)t_2)+(k-s)(t_0- st_1- (k- 2-s)t_2) \\
		&=s(t_1(1-k)+t_2(k-2))+(kt_0-k^2t_2+2kt_2) \\
		&=\frac{s((k-1)^{r-1}-1)^{m}}{k-1} (k(k-2)-((k-1)^{r-1}-1)^{\lambda-1})+(kt_0-k^2t_2+2kt_2).
	\end{align*}
	Since $k(k-2)-((k-1)^{r-1}-1)^{\lambda-1} \le 0$ when $r \geq 3, k \geq 3$, the minimum is attained at $s=k$. Hence,
	\begin{align*}
		&\min_{s \in \{0\} \cup \{ \brk{k}-\{k-1\}\}}s(t_0- (s- 1)t_1- (k- s)t_2)+(k-s)(t_0- st_1- (k- 2-s)t_2) \\
		&=k(t_0 - (k-1)t_1)\\
		&=k((k^{r-1}-1)^{m+\lambda}+(k-2)((k-1)^{k-1}-1)^{m})\\
		&=P(K_1 \vee \mathcal{H},k).
	\end{align*}
	For $k \in \brk{2}$, 
	$$0 \le P_{DP}(K_1 \vee \mathcal{H},k) \le P(K_1\vee \mathcal{H},k) =0.$$
	 Combining Lemma \ref{lem_odd}, we complete the proof of this theorem.
\end{proof}

		\section{Gluings of graphs}\label{sec:Glu}
	
	We now introduce a definition that generalizes the notion of a clique-sum for graphs. 
	Let $\mathcal{H}_1$ and $\mathcal{H}_2$ be hypergraphs such that $V(\mathcal{H}_1)\cap V(\mathcal{H}_2)\neq\varnothing$, and let $\mathcal{H}'=\mathcal{H}_1\cap\mathcal{H}_2.$
	Suppose that there exist polynomials $A_i(x)$ such that $P(\mathcal{H}_i,x)=P(\mathcal{H}',x)A_i(x)$ for $i\in \{1,2\}$. 
	For each $\theta \in \operatorname{Aut}(\mathcal{H}')$ the \emph{$(\mathcal{H}',\theta)$-sum} of $\mathcal{H}_1$ and $\mathcal{H}_2$, denote by $\mathcal{H}_1 \bigoplus_{\theta} \mathcal{H}_2$, is the hypergraph which obtained from disjoint copies of $\mathcal{H}_1$ and $\mathcal{H}_2$ by identifying each vertex $v$ of $\mathcal{H}'$ in $\mathcal{H}_1$ with $\theta(v)$ in $\mathcal{H}_2$.
	And the family
	$$
	\mathcal H_1\bigoplus\mathcal H_2 = \left\{ \mathcal{H}_1 \bigoplus_{\theta} \mathcal{H}_2: \theta \in \operatorname{Aut}(\mathcal{H}')
	\right\}.
	$$
	is called the \emph{$\mathcal{H}'$-sum family} of $\mathcal H_1$ and $\mathcal H_2$.
	
	The \emph{complete $r$-uniform hypergraph} on $p$ vertices, denoted by $K_p^r$, is the $r$-uniform hypergraph on a $p$-element vertex set in which every $r$-element subset forms an edge. 
	If $\mathcal{H}'=K_p^r,$ then we call $\mathcal{H}_1 \bigoplus \mathcal{H}_2$ a \emph{$K_p^r$-gluing} of $\mathcal{H}_1$ and $\mathcal{H}_2$. 
	Let $\bigoplus_{i=1}^n(\mathcal{H}_i,p,r)$ denote the family of all $K_p^r$-gluings of $\mathcal{H}_1,\mathcal{H}_2,\ldots,\mathcal{H}_n$, where $p\ge r$. 
	In particular, $\bigoplus_{i=1}^n(\mathcal{H}_i,1,1)$ denotes the family of all one-vertex gluings of $\mathcal{H}_1,\mathcal{H}_2,\ldots,\mathcal{H}_n$.
	
	Let $G_2$ is an arbitrary graph and $G_1 = C_{2k+2}.$ Suppose $G_1$ and $G_2$ share exactly two vertices and one edge. 
	Kaul and Mudrock \cite{KM21} proved that for $k\ge \max\{2,\chi(G_2)\}$
	$$P_{DP}(G,k) < P(G, k).$$
	We generalize this result as follows.

		\begin{theorem}
		Suppose that $\mathcal{H}_1$ is an arbitrary hypergraph and $\mathcal{H}_2$ is an $r$-uniform linear unicycle hypergraph consisting of $m + \iota$ edges, where $\iota$ denotes the length of the cycle. Assume that \(V(\mathcal{H}_1) \cap V(\mathcal{H}_2) \neq \emptyset\). Let $ \mathcal{H}' = \mathcal{H}_1 \cap \mathcal{H}_2$ and let $\mathcal{H} \in \mathcal{H}_1 \bigoplus \mathcal{H}_2$. If $k \ge \max \left\{ \chi(\mathcal{H}_1), \chi(\mathcal{H}_2) \right\}$ and $\iota$ is even, then $P_{DP}(\mathcal{H},k) < P(\mathcal{H},k).$
	\end{theorem}
	
	\begin{proof}
		
		Suppose that $e \in E(\mathcal{H}_2) \setminus E(\mathcal{H}_1)$. Since $\mathcal{H} \in \mathcal{H}_1 \bigoplus \mathcal{H}_2$, there exists a polynomial $A(k)$ such that
		$$
		P(\mathcal{H}-e, k)= \frac{P(\mathcal{H}_1,k)P(\mathcal{H}_2-e,k)}{A(k)}.
		$$
		and that 
		$$
		P(\mathcal{H},k) = \frac{P(\mathcal{H}_{1}, k)P(\mathcal{H}_2,k)}{A(k)}
		$$
		It follows from lemma~\ref{lem:2-1}and lemma~\ref{lem:2-3} that 
		$$
		P(\mathcal{H}-e, k)=\frac{k^{r-1}(k^{r-1}-1)^{m+\iota-1}P(\mathcal{H}_1,k)}{A(k)}
		$$
		and
		$$
		P(\mathcal{H},k)=\frac{\big((k^{r-1}-1)^{m+\iota}+(-1)^{\iota}(k-1)(k^{r-1}-1)^{m} \big)P(\mathcal{H}_1,k)}{A(k)}.
		$$
		When $\iota$ is even, we have that
		\begin{eqnarray*}
			\nonumber
			& &	\frac{k^{r-1}}{k^{r-1}-1}P(\mathcal{H},k)\\
			& = & \frac{k^{r-1}}{k^{r-1}-1}\frac{\big((k^{r-1}-1)^{m+\iota}+(-1)^{\iota}(k-1)(k^{r-1}-1)^{m}\big)P(\mathcal{H}_1,k)}{A(k)}		\\
			\nonumber
			& = & \big(k^{r-1}(k^{r-1}-1)^{m+\iota-1}+(-1)^{\iota}(k-1)k^{r-1}(k^{r-1}-1)^{m-1}\big) \frac{P(\mathcal{H}_1,k)}{A(k)} \\	
			& > & 
			\frac{k^{r-1}(k^{r-1}-1)^{m+\iota-1}P(\mathcal{H}_1,k)}{A(k)}\\
			&	= & P(\mathcal{H}-e, k).
		\end{eqnarray*}
		Thus, the result is clear from Theorem \ref{prop1.1}.
	\end{proof}
	
	Now we give two definitions which the first definition gives a natural way of splitting a cover of $\bigoplus^n_{i=1}(\mathcal{H}_i,p,r)$ into corresponding covers for each $\mathcal{H}_i$, and the second definition shows how to combine covers of $\mathcal{H}_i$ into a cover for a $\bigoplus^n_{i=1}(\mathcal{H}_i,p,r)$. 
	Whenever $\mathcal{F}$ is a $k$-fold cover of $\mathcal{H}$ and $\varPhi \subseteq \mathcal{F}$, we let $N(\varPhi, \mathcal{F})$ be the number of $\mathcal{F}$-colorings containing $\varPhi \subseteq \mathcal{F}$. 
	
	\begin{definition}
		Let $n\ge 2$, and let $\mathcal H_1,\mathcal H_2,\ldots,\mathcal H_n$ be vertex-disjoint hypergraphs. For each $i\in\brk n$, suppose that
		$$
		U_i=\{u_{i,1},u_{i,2},\ldots,u_{i,p}\}.
		$$
		 Let $\mathcal H$ be obtained by identifying $u_{1,j},u_{2,j},\ldots,u_{n,j}$ as a single vertex $u_j$ for each $j\in\brk p$.		
		For each $i\in\brk n$, let $\pi_i:V(\mathcal H_i)\rightarrow V(\mathcal H)$ be the natural identification map defined by
		$$
		\pi_i(x)=
		\begin{cases}
			u_j, & \text{if }x=u_{i,j}\text{ for some }j\in\brk p,\\
			x, & \text{otherwise}.
		\end{cases}
		$$
		Let $\mathcal F$ be a $k$-fold cover of $\mathcal H$. 
		For each $i\in\brk n$, the \emph{separated cover} of $\mathcal H_i$ induced by $\mathcal F$, denoted by $\mathcal F^{(i)}$, is defined as follows. 
		For each $e\in E(\mathcal H_i)$ and each map $\varphi\in\mathcal F_{\pi_i(e)}$, define the map $\varphi^{(\mathcal{H}_i)}$ on $e$ by
		$$
		\varphi^{(\mathcal{H}_i)}(x)=\varphi^{(\mathcal{H})}(\pi_i(x))\qquad\text{for every }x\in e.
		$$
		Then
		$$
		\mathcal F^{(i)}=\left\{\varphi^{(\mathcal{H}_i)}:\varphi\in\mathcal F_{\pi_i(e)},\;e\in E(\mathcal H_i)\right\}.
		$$
		We call $\mathcal F^{(i)}$ the separated $k$-fold cover of $\mathcal H_i$ obtained from $\mathcal F$.
	\end{definition}
	
	\begin{definition}\label{combine_covers}
		Let $n\ge 2$, and let $\mathcal H_1,\mathcal H_2,\ldots,\mathcal H_n$ be vertex-disjoint hypergraphs. For each $i\in\brk n$, let $\mathcal F^{(i)}$ be a $k$-fold cover of $\mathcal H_i$, and suppose that
		$$
		U_i=\{u_{i,1},u_{i,2},\ldots,u_{i,p}\}
		$$
		induces a copy of $K_p^r$ in $\mathcal H_i$.
		For each $2\le i\le n$, let
		$$
		g_i:\varphi_{\mathcal F^{(1)}}(U_1)\longrightarrow\varphi_{\mathcal F^{(i)}}(U_i)
		$$
		be a bijection, and set
		$$
		\mathbb G=(g_2,g_3,\ldots,g_n).
		$$
		Let $\mathcal H$ be the hypergraph obtained by identifying $u_{1,j},u_{2,j},\ldots,u_{n,j}$ as a single vertex $u_j$ for each $j\in\brk p$, and let $U=\{u_1,u_2,\ldots,u_p\}.$
		Let $g_1$ to be the identity map on $\varphi_{\mathcal F^{(1)}}(U_1)$. For each $i\in\brk n$, each $e\in E(\mathcal H_i)$, and each $\varphi\in\mathcal F^{(i)}_e$, let $\widehat{\varphi}$ denote the corresponding map on the image of $e$ in $\mathcal H$, obtained by preserving the colors outside $U$ and identifying the colors on $U_i$ according to $g_i$. More precisely,
		$$
		\widehat{\varphi}(x)=
		\begin{cases}
			\varphi(x), & \text{if } x\notin U,\\
			\left(g_i^{-1}\bigl(\varphi|_{U_i}\bigr)\right)(u_{1,j}), & \text{if } x=u_j\in U,
		\end{cases}
		$$
		where $u_{1,j}$ is identified with $u_j$ for each $j\in\brk p$.		
		The cover
		$$
		\mathcal F=\left\{\widehat{\varphi}:\varphi\in\mathcal F^{(i)}_e,\;e\in E(\mathcal H_i),\;i\in\brk n\right\}
		$$
		is called the \emph{$\mathbb G$-amalgamated $k$-fold cover} of $\mathcal H$ obtained from $\mathcal F^{(1)},\mathcal F^{(2)},\ldots,\mathcal F^{(n)}$.
	\end{definition}
	
	Let $\mathcal H_1$ and $\mathcal H_2$ be vertex-disjoint hypergraphs, and suppose that $\mathcal{H} = \bigoplus^2_{i=1}(\mathcal{H}_i,1,1)$. 
	Let $g_{2}: \varphi_{\mathcal{F}^{(1)}}(u_1) \to \varphi_{\mathcal{F}^{(2)}}(u_{2})$, where $g_{2}(\varphi^{(1)}_{\mathcal{F}^{(1)},1}(u_1))=\varphi^{(1)}_{\mathcal{F}^{(2)},2}(u_2)$ and $g_{2}(\varphi^{(1)}_{\mathcal{F}^{(1)},2}(u_1))=\varphi^{(1)}_{\mathcal{F}^{(2)},1}(u_2)$.
	Let $\mathbb{G} = \{g_2\}$, and below is an example of Definition \ref{combine_covers}.
	\medskip
	
	\begin{minipage}{\textwidth}
		\begin{minipage}[t]{0.48\textwidth}
			\centering
			\includegraphics[width=0.6\linewidth]{"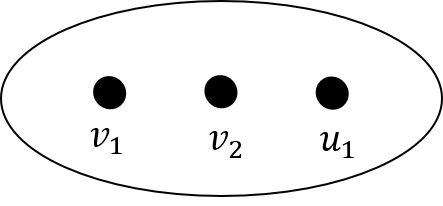"}
			
			\medskip
			
			\begin{tabular}{c c c c c}
				\hline
				& $v_1$ & $v_2$ & $u_1$ & $v_3$ \\
				\hline
				$\varphi^{(1)}_1$ & 1 & 1 & 1 &    \\
				$\varphi^{(1)}_2$ & 2 & 2 & 2 &   \\
				\hline
				$\varphi^{(2)}_1$ &  &  & 1 & 2   \\
				$\varphi^{(2)}_2$ &  &  & 2 & 1  \\
				\hline
			\end{tabular}
			\captionof{figure}{$\mathcal{H}_1$ and a $2$-fold cover of $\mathcal{H}_1$}
			\vspace{1em}
			
		\end{minipage}
		\hfill
		\begin{minipage}[t]{0.48\textwidth}
			\centering
			\includegraphics[width=0.6\linewidth]{"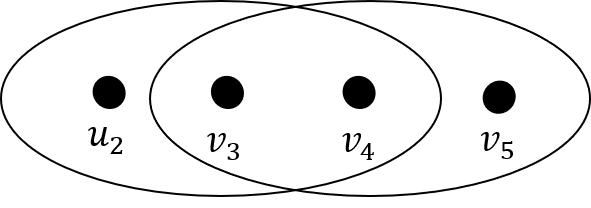"}
			
			\medskip
			
			\begin{tabular}{c c c c c}
				\hline
				& $u_2$ & $v_4$ & $v_5$ & $v_6$   \\
				\hline
				$\varphi^{(1)}_1$ & 1 & 1 & 2 &   \\
				$\varphi^{(1)}_2$ & 2 & 2 & 1 &   \\
				\hline
				$\varphi^{(1)}_1$ &  & 1 & 1 & 1  \\
				$\varphi^{(1)}_2$ &  & 2 & 2 & 2  \\
				\hline
			\end{tabular}
			\captionof{figure}{$\mathcal{H}_2$ and a $2$-fold cover of $\mathcal{H}_1$}
		\end{minipage}
		
		\begin{minipage}[t]{0.48\textwidth}
			\centering
			\includegraphics[width=0.8\linewidth]{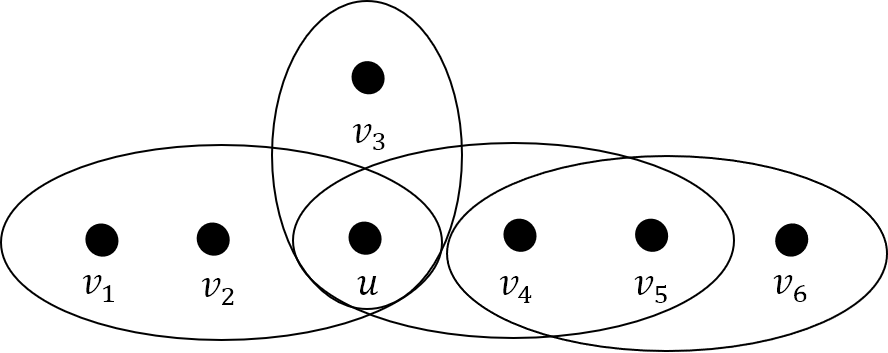}
			
			\medskip
			
			\begin{tabular}{c c c c c c c c}
				\hline
				& $v_1$ & $v_2$ & $u$& $v_3$ & $v_4$ & $v_5$ &$v_6$\\
				\hline
				$\varphi^{(1)}_1$ & 1 & 1 & 1 &  &  &  & \\
				$\varphi^{(1)}_2$ & 2 & 2 & 2 &  &  &  & \\
				\hline
			\end{tabular}
		\end{minipage}
		\hfill
		\begin{minipage}[t]{0.48\textwidth}
			\begin{tabular}{c c c c c c c c}
				\hline
				& $v_1$ & $v_2$ & $u$& $v_3$ & $v_4$ & $v_5$ &$v_6$\\
				\hline
				$\varphi^{(2)}_1$ &  &  & 1 & 2 &  &  & \\
				$\varphi^{(2)}_2$ &  &  & 2 & 1 &  &  & \\
				\hline
				$\varphi^{(3)}_1$ &  &  & 2 & & 1 & 2 &   \\
				$\varphi^{(3)}_2$ &  &  & 1 && 2 & 1 &   \\
				\hline
				$\varphi^{(4)}_1$ & &  &  &  & 1 & 1 & 1   \\
				$\varphi^{(4)}_2$ & & &  &  & 2 & 2 & 2   \\
				\hline
			\end{tabular}
		\end{minipage}
		\captionof{figure}{$\mathcal{H}$ and the $\mathbb{G}$-amalgamated-fold cover of $\mathcal{H}$}
		\label{example}
	\end{minipage}
	\begin{lemma}\label{lem_glue}
		Suppose that $ \mathcal{H}_1, \dots, \mathcal{H}_n$ are vertex disjoint hypergraphs for some $n \geq 2$ where $u_i \in V(\mathcal{H}_i)$ for each $i \in \brk{n} $. 
		Suppose that $\mathcal{F}^{(i)}$ is a $k$-fold cover of $\mathcal{H}_i$ for each $ i \in \brk{n}$. Suppose that $ g_i : \varphi^{(\mathcal{H}_1)}(u_1) \to \varphi^{(\mathcal{H}_i)}(u_i) $ is a bijection for each $ 2 \leq i \leq n $, and $ \mathbb{G} = (g_2, \dots, g_n) $. 
		Let $\mathcal{H}$ be the hypergraph obtained by identifying $ u_1, \dots, u_n $ as the same vertex $ u $. 
		Let $f: V(\mathcal{H}) \to \brk{k}$ and 
		$$ D = \sum_{j=1}^n N(f(u_1)=j, \mathcal{F}^{(1)}) \prod_{i=2}^n N(g_i(f(u_1)=j), \mathcal{F}^{(i)}).$$ 
		Then 
		$$ P_{DP}(\mathcal{H}, \mathcal{F}) = D $$ where $ \mathcal{F}$ is the $ \mathbb{G}$-amalgamated cover of $\mathcal{H}$ obtained from $ \mathcal{F}^{(1)}, \dots, \mathcal{F}^{(n)} $. 
		Consequently, 
		$$ P_{DP}(\mathcal{H}, k) \leq D.$$
	\end{lemma}
	\begin{proof}
		Let $\mathcal F$ be the $\mathbb G$-amalgamated $k$-fold cover of $\mathcal H$ obtained from $\mathcal F^{(1)},\ldots,\mathcal F^{(n)}$. Clearly,
		$$
		P_{DP}(\mathcal H,\mathcal F)=\sum_{j=1}^{k}N(f(u)=j,\mathcal F).
		$$
		Fix $j\in\brk k$. For each $2\le i\le n$, let $z_i$ be the unique color of $u_i$ such that
		$$
		g_i(j)=z_i.
		$$
		Let $\mathcal I_1$ be the set of all $\mathcal F^{(1)}$-colorings $f_1$ of $\mathcal H_1$ satisfying $f_1(u_1)=j$. For each $2\le i\le n$, let $\mathcal I_i$ be the set of all $\mathcal F^{(i)}$-colorings $f_i$ of $\mathcal H_i$ satisfying $f_i(u_i)=z_i$.
		Define
		$$
		h:\mathcal I_1\times\mathcal I_2\times\cdots\times\mathcal I_n\longrightarrow\mathcal I
		$$
		by $h(f_1,f_2,\ldots,f_n)=f$, where
		$$
		f(u)=j
		$$
		and
		$$
		f(v)=f_i(v)\qquad\text{for every }v\in V(\mathcal H_i)\setminus\{u_i\}\text{ and }i\in\brk n.
		$$
		Since $z_i=g_i(j)$ for each $2\le i\le n$, the colors at $u_1,u_2,\ldots,u_n$ match under the maps in $\mathbb G$. After identifying $u_1,u_2,\ldots,u_n$ as $u$, they all correspond to the color $j$ at $u$. Hence $h$ is well defined.
		Moreover, $h$ is a bijection. For each $f\in\mathcal I$, the coloring $f_i$ is uniquely determined by $f$ on $V(\mathcal H_i)\setminus\{u_i\}$ and by the color of $u_i$. Here $f_1(u_1)=j$, and $f_i(u_i)=z_i$ for each $2\le i\le n$. Thus
		$$
		h^{-1}(f)=(f_1,f_2,\ldots,f_n),
		$$
		where
		$$
		f_1(u_1)=j,\qquad f_i(u_i)=z_i\quad\text{for }2\le i\le n,
		$$
		and
		$$
		f_i(v)=f(v)\qquad\text{for every }v\in V(\mathcal H_i)\setminus\{u_i\}.
		$$
		Consequently,
		$$
		N(f(u)=j,\mathcal F)=|\mathcal I|=\prod_{i=1}^{n}|\mathcal I_i|
		=N(f_1(u_1)=j,\mathcal F^{(1)})\prod_{i=2}^{n}N(f_i(u_i)=z_i,\mathcal F^{(i)}).
		$$
		Therefore, 
		$$P_{DP}(\mathcal{H}, \mathcal{F}) = D.$$
		This is clear that 
		$$P_{DP}(\mathcal{H}, k) \le D.$$
	\end{proof}
	
	The following result is a generalization of Theorem \ref{result 3} obtained by Becker et al.
	\begin{theorem}\label{th-2-1}
		Let $n\ge 2$, and let $\mathcal H_1,\ldots,\mathcal H_n$ be pairwise vertex-disjoint hypergraphs. Suppose that $\mathcal H\in\bigoplus_{i=1}^{n}(\mathcal H_i,1,1)$. Then, for every integer $k\ge 1$,
		$$
		P_{DP}(\mathcal H,k)\le \frac{\prod_{i=1}^{n}P_{DP}(\mathcal H_i,k)}{k^{n-1}}.
		$$
	\end{theorem}
	
	\begin{proof}
		Let $\mathcal{H}$ be the hypergraph obtained by identifying $ u_1, \dots, u_n $ as the same vertex $ u $, where $u_i \in V(\mathcal{H}_i)$ for each $i \in \brk{n}$.The proof is by induction on $n$.
		First suppose that $n=2$. For each $i\in\{1,2\}$, choose a $k$-fold cover $\mathcal F^{(i)}$ of $\mathcal H_i$ such that
		$$
		P_{DP}(\mathcal H_i,\mathcal F^{(i)})=P_{DP}(\mathcal H_i,k).
		$$
		For each $j\in\brk{k}$, set
		$$
		a_j=N(f(u_1)=j,\mathcal F_1),\qquad b_j=N(f(u_2)=j,\mathcal F_2).
		$$
		Then
		$$
		\sum_{j=1}^{k}a_j=P_{DP}(\mathcal H_1,k),\qquad 
		\sum_{j=1}^{k}b_j=P_{DP}(\mathcal H_2,k).
		$$
		
		For each $d\in\brk{k}$, define a bijection $g_d$ from the color set at $u_1$ to that at $u_2$ by
		$$
		g_d(j)=
		\begin{cases}
			j+d-1, & \text{if }j+d-1\le k,\\
			j+d-1-k, & \text{if }j+d-1>k,
		\end{cases}
		\qquad j\in\brk{k}.
		$$
		Let $\mathcal F_d$ be the $g_d$-amalgamated $k$-fold cover of $\mathcal H$ obtained from $\mathcal F^{(1)}$ and $\mathcal F^{(2)}$. By Lemma \ref{lem_glue},
		$$
		P_{DP}(\mathcal H,\mathcal F_d)=\sum_{j\in\brk{k}}a_jb_{g_d(j)}=:D_{g_d}.
		$$
		Thus $P_{DP}(\mathcal H,k)\le D_{g_d}$ for every $d\in\brk{k}$.
		Then,
		$$
		\begin{aligned}
			\sum_{d=1}^{k}D_{g_d}
			&=\sum_{d=1}^{k}\sum_{j=1}^{k}a_jb_{g_d(j)}\\
			&=\sum_{j=1}^{k}a_j\left(\sum_{d=1}^{k}b_{g_d(j)}\right)\\
			&=\left(\sum_{j=1}^{k}a_j\right)\left(\sum_{q=1}^{k}b_q\right)\\
			&=P_{DP}(\mathcal H_1,k)P_{DP}(\mathcal H_2,k).
		\end{aligned}
		$$
		Hence there exists $d_0\in\brk{k}$ such that
		$$
		P_{DP}(\mathcal H,k)\le D_{g_{d_0}}\le \frac{1}{k}\sum_{d\in\brk{k}}D_{g_d}=\frac{P_{DP}(\mathcal H_1,k)P_{DP}(\mathcal H_2,k)}{k}.
		$$
		This proves the case $n=2$.
		
		Now let $n\ge 3$, and suppose that the result holds for $n-1$ hypergraphs. Let $\mathcal H'$ be obtained from $\mathcal H_1,\ldots,\mathcal H_{n-1}$ by identifying $u_1,\ldots,u_{n-1}$ as a single vertex $u'$. By the induction hypothesis,
		$$
		P_{DP}(\mathcal H',k)\le \frac{\prod_{i=1}^{n-1}P_{DP}(\mathcal H_i,k)}{k^{n-2}}.
		$$
		Since $\mathcal H$ is obtained from $\mathcal H'$ and $\mathcal H_n$ by identifying $u'$ with $u_n$, the case $n=2$ gives
		$$
		P_{DP}(\mathcal H,k)\le \frac{P_{DP}(\mathcal H',k)P_{DP}(\mathcal H_n,k)}{k}\le \frac{\prod_{i=1}^{n}P_{DP}(\mathcal H_i,k)}{k^{n-1}}.
		$$
		This completes the proof.
	\end{proof}
	
	\begin{proposition}\label{DP-low}
		Let $\mathcal H=(X,E)$ be a hypergraph with $|X|=n$ and $|e|\ge 2$ for every $e\in E$. Fix an ordering $v_1,\ldots,v_n$ of $X$. For each $i\in\brk n$, let
		$$
		d_i=\left|\left\{e\in E:v_i\in e,\;e\setminus\{v_i\}\subseteq\{v_1,\ldots,v_{i-1}\}\right\}\right|,
		$$
		so that $d_i$ counts the edges whose last vertex in the ordering is $v_i$. Set $D=\max_{i\in\brk n}d_i$. Then, for every integer $k\ge D+1$,
		$$
		P_{DP}(\mathcal H,k)\ge \prod_{i=1}^{n}(k-d_i).
		$$
	\end{proposition}
	
	\begin{proof}
		Let $\mathcal F$ be an arbitrary perfect $k$-fold cover of $\mathcal H$. We color the vertices in the order $v_1,\ldots,v_n$. We now color the vertices by the order.
		
		Under the ordering $v_1,\ldots,v_n$, the color of $v_1$ may be chosen arbitrarily.
		Fix $i\in\brk n$, and suppose that $v_1,\ldots,v_{i-1}$ have already been colored. For each edge $e$ counted by $d_i$, all vertices in $e\setminus\{v_i\}$ have already been colored before we color $v_i$. 
		Since the maps in $\mathcal F_e$ are pairwise disjoint, at most one map $\varphi\in\mathcal F_e$ agrees with the colors already assigned to $e\setminus\{v_i\}$. If such a map exists, we avoid it by choosing $f(v_i)\ne\varphi(v_i)$. Thus, each edge counted by $d_i$ forbids at most one color for $v_i$.
		
		Since there are exactly $d_i$ such edges, at most $d_i$ colors are forbidden for $v_i$. 
		Thus, at least $k-d_i$ colors remain available for $v_i$. 
		This holds for every partial coloring obtained in the process.
		Every edge is checked when its last vertex is colored. 
		Therefore, every complete coloring produced in this way is an $\mathcal F$-coloring. 
		Since different choices give different colorings, the total number of $\mathcal F$-colorings is at least $\prod_{i=1}^{n}(k-d_i).$
		Taking the minimum over all perfect $k$-fold covers of $\mathcal H$ gives the desired result.
	\end{proof}
	
	\begin{lemma}\label{lemma4-2}
		Let $n\ge 2$, and let $\mathcal H_1,\mathcal H_2,\ldots,\mathcal H_n$ be pairwise vertex-disjoint hypergraphs. For each $i\in\brk n$, suppose that
		$$
		U_i=\{u_{i,1},u_{i,2},\ldots,u_{i,p}\}
		$$
		induces a copy of $K_p^r$ in $\mathcal H_i$, where $p\ge r$. Let $\mathcal H\in\bigoplus_{i=1}^{n}(\mathcal H_i,p,r)$ be obtained by identifying $u_{1,j},u_{2,j},\ldots,u_{n,j}$ as a single vertex $u_j$ for each $j\in\brk p$, and let
		$$
		U=\{u_1,u_2,\ldots,u_p\}.
		$$
		
		Suppose that for each $i\in\brk n$, every $k$-fold cover $\mathcal F_{\mathcal H_i}$ of $\mathcal H_i$ satisfies
		$$
		N(f_i,\mathcal F_{\mathcal H_i})\ge m_i
		$$
		for every $\mathcal F_{\mathcal H_i}$-coloring $f_i$ of $U_i$. If
		$$
		k\ge \binom{p-1}{r-1}+1,
		$$
		then
		$$
		P_{DP}(\mathcal H,k)\ge
		\left(\prod_{q=1}^{p}(k-d_q)\right)
		\left(\prod_{i=1}^{n}m_i\right),
		$$
		where
		$$
		d_q=
		\begin{cases}
			0, & \text{if }q\in\brk{r-1},\\
			\binom{q-1}{r-1}, & \text{if }q\in\brk p\setminus\brk{r-1}.
		\end{cases}
		$$
	\end{lemma}
	
	\begin{proof}
		Let $\mathcal F$ be an arbitrary $k$-fold cover of $\mathcal H$, and for each $i\in\brk n$, let $\mathcal F^{(i)}$ be the separated $k$-fold cover of $\mathcal H_i$ obtained from $\mathcal F$.
		
		Let $\mathcal F_U$ be the restriction of $\mathcal F$ to the copy of $K_p^r$ induced by $U$. Fix an $\mathcal F_U$-coloring $f_U$ of $U$. For each $i\in\brk n$, define the corresponding coloring $f_i$ of $U_i$ by
		$$
		f_i(u_{i,q})=f_U(u_q)
		$$
		for every $q\in\brk p$.
		
		For each $i\in\brk n$, let $\mathcal I_i$ be the set of all $\mathcal F^{(i)}$-colorings of $\mathcal H_i$ containing $f_i$, and let $\mathcal I$ be the set of all $\mathcal F$-colorings of $\mathcal H$ containing $f_U$. By the assumption,
		$$
		|\mathcal I_i|=N(f_i,\mathcal F^{(i)})\ge m_i
		$$
		for every $i\in\brk n$.
		
		Define
		$$
		h:\mathcal I_1\times\mathcal I_2\times\cdots\times\mathcal I_n\longrightarrow\mathcal I
		$$
		by $h(f_1,\ldots,f_n)=f$, where
		$$
		f(u_j)=f_U(u_j)
		$$
		for every $j\in\brk p$, and
		$$
		f(v)=f_i(v)
		$$
		for every $v\in V(\mathcal H_i)\setminus U_i$ and $i\in\brk n$.
		
		The map $h$ is well defined, since all the colorings $f_i$ agree on the identified vertices through $f_U$. Moreover, $h$ is a bijection. Indeed, every coloring in $\mathcal I$ restricts uniquely to a coloring in $\mathcal I_i$ for each $i\in\brk n$, and conversely any choice of colorings from $\mathcal I_1,\ldots,\mathcal I_n$ determines a unique coloring in $\mathcal I$. Hence
		$$
		N(f_U,\mathcal F)=|\mathcal I|=\prod_{i=1}^{n}|\mathcal I_i|\ge \prod_{i=1}^{n}m_i.
		$$
		
		Now order the vertices of $U$ as $u_1,u_2,\ldots,u_p$. Since $U$ induces a copy of $K_p^r$, the number of edges whose last vertex in this ordering is $u_q$ is
		$$
		d_q=
		\begin{cases}
			0, & \text{if }q\in\brk{r-1},\\
			\binom{q-1}{r-1}, & \text{if }q\in\brk p\setminus\brk{r-1}.
		\end{cases}
		$$
		By Proposition \ref{DP-low}, the number of $\mathcal F_U$-colorings of $U$ is at least
		$$
		\prod_{q=1}^{p}(k-d_q).
		$$
		
		For each such coloring $f_U$, there are at least $\prod_{i=1}^{n}m_i$ $\mathcal F$-colorings of $\mathcal H$ containing $f_U$. Since distinct colorings of $U$ give disjoint sets of $\mathcal F$-colorings of $\mathcal H$, we obtain
		$$
		P_{DP}(\mathcal H,\mathcal F)\ge
		\left(\prod_{q=1}^{p}(k-d_q)\right)
		\left(\prod_{i=1}^{n}m_i\right).
		$$
		Since $\mathcal F$ was arbitrary, it follows that
		$$
		P_{DP}(\mathcal H,k)\ge
		\left(\prod_{q=1}^{p}(k-d_q)\right)
		\left(\prod_{i=1}^{n}m_i\right).
		$$
	\end{proof}
	
	From Theorem \ref{th-2-1} and lemma \ref{lemma4-2}, we can easily get the corollary.
	\begin{corollary}
		Suppose that $ \mathcal{H}_1, \dots, \mathcal{H}_n$ are vertex disjoint hypergraphs for some $n \geq 2$ where $u_i \in V(\mathcal{H}_i)$ for each $i \in \brk{n} $. 
		Let $\mathcal{H} \in \bigoplus_{i=1}^{n}(\mathcal{H}_i,1,1)$.
		For $i\in \brk{n}$, let $\mathcal{F}_{\mathcal{H}_i}$ be an arbitrary $k$-fold cover of $\mathcal{H}_i$ satisfying $N(f(u_i)=s,\mathcal{F}_{\mathcal{H}_i}) \geq P_{DP}(\mathcal{H}_i,k)/k$, where $s \in \brk{k}$. Then
		$$
		P_{DP}(\mathcal{H},k)=\frac{\prod_{i=1}^{n} P_{DP}(\mathcal{H}_i,k)}{k^{n-1}}
		$$
	\end{corollary}

	\vskip 5mm
	\noindent{\bf Acknowledgements }
	\vskip 5mm
	
	This work is supported by National Natural Science Foundation of
	China under Grant
	No. 12371340.

\end{document}